\documentclass{amsart}
\usepackage{amssymb}
\usepackage{mathrsfs}
\usepackage{graphicx}
\usepackage{comment}

\newcommand{\ve}{\varepsilon}

\newcommand{\be}{\begin{equation}}
\newcommand{\ee}{\end{equation}}
\newcommand{\ba}{\begin{align}}
\newcommand{\ea}{\end{align}}

\newcommand{\abs}[1]{\lvert#1\rvert}

\newtheorem{theorem}{Theorem}[section]

\newtheorem{lemma}{Lemma}[section]

{\begin{list}{}{%
\settowidth{\labelwidth}{\textsf{{\it #1.}}}%
\setlength{\labelsep}{4mm}%
\setlength{\leftmargin}{\labelwidth}%
\addtolength{\leftmargin}{\labelsep}%
}}%
{\end{list}}

\def\beq{\begin{equation}}\def\enq{\end{equation}}

{\begin{list}{}{%
\settowidth{\labelwidth}{\textsf{{\it #1.}}}%
\setlength{\labelsep}{2mm}%
\setlength{\leftmargin}{\labelwidth}%
\addtolength{\leftmargin}{\labelsep}%
\addtolength{\leftmargin}{4mm}%
\setlength{\itemsep}{6pt}%
\setlength{\listparindent}{0pt}%
\setlength{\topsep}{3pt}%
}}%

\title{Integer circulant determinants of order 15}

\author[B. Paudel]{Bishnu Paudel}
\address{ Department of Mathematics\\
         Kansas State University\\
         Manhattan, KS 66506, USA}
\email{bpaudel@ksu.edu, pinner@math.ksu.edu}

\author[C. Pinner]{Chris Pinner}

\thanks{This began as a K-State I-Center undergrad research project for Gonzalo Rodrigues Sanabria}

\date{\today}
\begin{document}

\begin{abstract} We consider the values taken by  $n\times n$ circulant determinants with integer entries
when $n$ is the product of two distinct odd primes $p,q$.
These correspond to the integer group determinants for $\mathbb Z_{pq}$, the cyclic group of order $pq$.
We show that $p^2$ and $q^2$ are not determinants (more generally we show that  the classic necessary divisibility conditions are never sufficient when $n$ contains at least two odd primes). We obtain a complete description of the integer group determinants for  $\mathbb Z_{15}$ (the smallest unresolved group) 
and  partial results for general $n=3p.$
\end{abstract}

\maketitle

\section{Introduction}\label{secIntroduction}
We recall that a {\em circulant determinant} is one where successive rows arise by a cyclic shift of the previous row one step to the right
$$ D(a_0,\dots ,a_{n-1}):=\det \begin{pmatrix} a_0 & a_1 &  \dots & a_{n-1} \\
a_{n-1} & a_0 & \cdots  & a_{n-2}\\
\vdots & \vdots  &  & \vdots \\
a_1 & a_2 & \cdots & a_0 \end{pmatrix}. $$ 
At the 1977  meeting of the American Mathematical Society in Hayward, California, Olga Taussky-Todd
asked which integers can be obtained as an integral $n\times n$ circulant determinant:
$$ S_n:=\{ D(a_0,\ldots ,a_{n-1}) \; : \; (a_0,\ldots ,a_{n-1})\in \mathbb Z^n\}. $$
For a finite group $G$, and $|G|$ variables $a_g$, $g\in G$, the group determinant is defined to be the polynomial  obtained by taking the determinant of the
matrix whose $ij$th entry is $a_{g_ig_{j}^{-1}}$. One can similarly ask what integer values the group determinants 
take when the variables $a_g$ are all integral.  The group determinant polynomial determines the group \cite{Formanek}, but
it remains open whether the integer values determine the group. 
Determining $S_n$ is plainly  the same as determining the integer group determinants  in the special case of the cyclic group  
$$\mathbb Z_n =\{0,1,\ldots ,n-1 \}\text{ mod } n. $$
In \cite{smallgps} a complete description of the integer group determinants was obtained for all groups with $|G|\leq 14$.
Partial results have been obtained for other families of finite groups, \cite{dihedral,Stian,dilum,pgroups,2gp,Cid2,S4}.
Here we consider the smallest unresolved group $\mathbb Z_{15}$, the  $15\times 15$ integer circulant determinants $S_{15}$.
As observed by Newman \cite{Newman1} (or using characters to factor the group determinant, see for example \cite{Conrad}),
we can write
$$  D(a_0,\dots ,a_{n-1})=M_{n}\left(a_0+a_1x+\cdots +a_{n-1}x^{n-1}\right) $$
where for a polynomial $F(x)$ in $\mathbb Z[x]$  we define
\be \label{DefM}  M_n(F) :=\prod_{j=1}^{n} F(\omega_n^j),\;\; \omega_n:= e^{2\pi i/n}. \ee
It will often be convenient to break this down as a product of integer norms
\be \label{defM_n} M_n(F):=\prod_{d\mid n} N_d(F),\;\; \; \;N_d(F): =\prod_{\stackrel{j=1}{\gcd(j,d)=1}}^d F(\omega_d^j), \ee
by dividing the $n$th roots of unity into  the various  primitive $d$th roots of unity.
We can think of $M_n(F)$ as the resultant of $F$ with $x^n-1$ and the $N_d(F)$ the resultants with its irreducible factors,
the $d$th cyclotomic polynomials:
$$  \Phi_n(x):=\prod_{\stackrel{j=1}{\gcd(j,n)=1}}^n \left(x-\omega_n^j\right), \;\;\;\; x^n-1=\prod_{d\mid n}\Phi_d(x). $$     
Note, we can recover a circulant determinant from a polynomial of degree $n$ or more by reducing mod $(x^n-1)$.
See \cite{Norbert,Norbert2,Lind} for a discussion of $\frac{1}{n}\log |M_n(F)|$ as a $\mathbb Z_n$ group generalization of the classical logarithmic Mahler Measure, \cite{Lalin} for an alternative approach. In \cite{Pigno1} the smallest non-trivial value, the counterpart of the classical Lehmer Problem \cite{Lehmer},  was obtained  for all cyclic groups of order less than $892,371,480$. 

Trivially $S_n$ is closed under multiplication (from \eqref{defM_n} or by multiplying 
elements $\sum_{g\in G} a_g g$ in the group ring).
Other elementary properties were obtained in Newman \cite{Newman1} \& Lacquer \cite{Laquer}. For example,
\be \label{basic}  \{ m\in \mathbb Z\;\; : \;\; \gcd(m,n)=1\} \cup n^2\mathbb Z \subset S_n,  \ee
to be explicit, $M_n(-x)=-1$ and  if $\gcd(m,n)=1$
$$ M_n\left( \prod_{p^{\alpha}\parallel m} \Phi_p(x)^{\alpha}\right)=|m|,\;\;\;\; M_n\left( 1-x +k \left(\frac{x^n-1}{x-1}\right)\right)=kn^2. $$ 
We shall frequently use that the absolute value of the  resultant of two cyclotomic polynomials $\Phi_k(x)$ and $\Phi_s(x)$, $k>s$, is $p^{\phi(s)}$ if $k=sp^{\alpha}$ and  one otherwise   (see for example   \cite{Apostol,ELehmer}).
Newman and Laquer also obtained the divisibility restrictions
\be \label{div} t\in S_n, \;\; p\mid t,\;\; p^{\alpha}\parallel n \;\; \Rightarrow \;\; p^{\alpha+1}\mid t. \ee
For odd primes $p$  this led them to a precise description of $S_p$ and $S_{2p},$
$$ S_p= \{ p^a m\;\; :\;\; \gcd(m,p)=1,\;\; a=0 \text{ or } a\geq 2\}   $$
and
$$ S_{2p}=\{2^ap^bm\;\; : \; \; \gcd(m,2p)=1,\;\; a=0 \text{ or } a\geq 2, \;\;b=0 \text{ or } b\geq 2\},$$
with \eqref{div} being an iff condition. While \eqref{div} is always sharp, in the sense that 
$$ p^{\alpha}\parallel n \;\; \Rightarrow \;\; p^{\alpha+j}\parallel M_n(x-1+p^j) \text{ for all } j\geq 1, $$
we will not in general obtain all multiples of $p^{\alpha+1}$. For example Newman \cite{Newman2} showed that $p^{\alpha+1}\not\in S_{p^{\alpha}}$ for any $\alpha\geq 2$ when $p\geq 5$. Here we show:

\begin{theorem} \label{generalisation}
Suppose that $n$ is divisible by two distinct odd primes $p$ and $q$ with $p^{\alpha}\parallel n$. Then $p^{\alpha+1}$ is not in $S_n.$
\end{theorem}

We concentrate on the case  $n=pq$ where  $p$ and $q$ are distinct odd primes. 
From \eqref{div} the only determinants, in addition to \eqref{basic}, must be of 
the form $p^2m,$ $q\nmid m$ and $q^2m$, $p\nmid m$. But from Theorem \ref{generalisation} not all integers of this form with be determinants, for example $p^2$ and $q^2$ are not themselves determinants. 
In all cases of $n=3p$ (and a couple of cases of $n=5p$ that we tested computationally) we do though obtain all such multiples of $p^3$ and $q^3$:

\begin{theorem}\label{cubesthm}  For $n=3p$ or 35 or 55 we have
\be \label{cubes} \{p^3m\; \:\; \gcd(m,q)=1\} \cup \{ q^3m\; :\; \gcd(m,p)=1\}\subset S_{pq}. \ee
\end{theorem}
This follows immediately from the lemmas in Section \ref{p^3} below.

\vspace{1ex}
\noindent
{\bf Question 1.}  Does  \eqref{cubes} in fact hold for all $n=pq$?

\vspace{1ex}
Determining precisely which multiples $3^2m$ and $p^2m$, $\gcd(m,30)=1$,  are determinants for $n=3p$ would seem to be a much harder 
problem.
Our goal here is only to make this explicit for $n=15$. It helps here that we have uniqueness of factorisation in all the 
underlying cyclotomic extensions $\mathbb Z [\omega_{n}]$, $n=3,5$ and 15 (see for example Washington \cite[Theorem 11.1]{Washington}).  It will require us to divide  the primes $p\equiv 1$ mod 15 into two classes according to  their representation as a $15$-norm; we shall see  that every  $p\equiv 1$ mod 15 can be written in the form
$$ p= N_{15}\left(  (x^5-1)\pm x^j \Phi_3(x) B(x) +(x-1)\Phi_3(x)\Phi_5(x)g(x)\right) $$
for some $g(x)$ in $\mathbb Z[x]$, $0\leq j<15$,  and either  $B(x)=1$ (we shall call these primes {\em good}) or $B(x)=(x-1)$
(we shall call these primes {\em bad}).

\begin{theorem}\label{Z15}
The determinants in  $S_{15}$ take the form $15^2\mathbb Z$ or $m$ or $3^{t}m,5^tm$ with $t\geq 3$,
or $3^2k$, $5^2k$ with
\begin{itemize}
\item[(i)] $k=mp$, $p\equiv 7,11$ or $13$ mod $15$, or  a ``good" $p\equiv 1$ mod 15, or
\item[(ii)] $k=mp^2$, $p\equiv 4$ mod 15,  or
\item[(iii)] $k=mp^4$, $p\equiv 2$ or $8$ mod $15$,
\end{itemize}
where $m$ can be any integer coprime to 15 and $p$ denotes a prime.

\end{theorem} 

The downside is that it is not immediately obvious which primes $p\equiv 1$ mod 15 are good.
For example the {\em good} $p\leq 5000$, for  which we can obtain $3^2p$ and $5^2p$, are
\begin{align*} & 31,151,181,421,601,661,691,751,811,1051,1171,1231,1291,1321,1531,1621,1741,\\
& 1831,1861,2221,2281,2371, 2521, 2551, 2971, 3061, 3181,  3271, 3301, 3361, 3391, 3511, \\
& 3691, 4051, 4111, 4201, 4231, 4561, 4621, 4831, 4951.
\end{align*}
and the {\em bad} primes
\begin{align*} &  61,211,241,271,331,541,571,631,991,1021,1201,1381,1471,1801,1951,2011,2131,\\
& 2161,2251,2311, 2341, 2671, 2791, 2851, 3001, 3121, 3331, 3541, 3571, 3631, 3931, 4021, \\
& 4261, 4441, 4591, 4651, 4801, 4861.
\end{align*}

The complexity that we encountered for $n=3p$ for $p=5$ did not make us want to attempt this for larger $p$,
although $\mathbb Z [\omega_{3p}]$ does have uniqueness of factorisation for $p=7$ or 11.

\section{Proof of Theorem \ref{generalisation}}\label{p^2}

\begin{lemma} \label{Kron} If $u(x)$ is in $\mathbb Z[x]$ and $u(\omega_{n})$ is a unit in $\mathbb Z[\omega_{n}]$ then
$u(\omega_{n}^{-1})=\ve \omega_{n}^J u(\omega_{n})$ for some integer $J\geq 0$ and $\ve=\pm 1$. 
If $n=p$ is prime then $\ve=+1$.

\end{lemma} 

\begin{proof} Since $u(\omega_{n})$ is a unit we know that $\alpha=u(\omega_{n}^{-1})/u(\omega_{n})$ is
an algebraic integer with $|\alpha|=1$. The same is true for all its conjugates and hence  $\alpha$, by Kronecker's Theorem \cite{Kronecker}, must be a root of unity in $\mathbb Z[\omega_{n}].$ So  $\alpha=\pm \omega_{n}^J$ for some integer $J\geq 0$.

If $n=p$ is prime, then $u(\omega_p)^p,u(\omega^{-1}_p)^p\equiv u(1)$ (mod $p$), and we can rule out $\ve=-1,$ since then
$ 0=u(\omega_p^{-1})^p+ u(\omega_p)^p \equiv 2u(\omega_p)^p \text{ mod }p,$
but $|u(\omega_p)|_p=1$.
\end{proof}

\begin{proof}[Proof of Theorem \ref{generalisation}] Since for $\gcd(s,t)=1$ and $F$ in $\mathbb Z[x]$  we have
$$ M_{st}(F)=M_s(G),\;\;\; G(x)=\prod_{j=1}^t F(x\omega_t^j)\in \mathbb Z[x], $$
we can assume that $n=p^{\alpha}q^{\ell}$. Suppose that $F(x)$ is a polynomial in $\mathbb Z[x]$ of degree $d$ with $M_{n}(F)=p^{\alpha+1}$.

Since for $\gcd(r,p)=1$ we have $N_{rp^j }(F)\equiv N_r(F)^{\phi(p^j)}$ mod $p,$ we readily see that $M_{p^kq^{\ell}}(F)=p^{\alpha+1}$
can only happen when $N_{p^j q^{\beta}}(F)=p$ for some $0\leq \beta \leq \ell$ and all $0\leq j\leq \alpha,$ with $N_{p^jq^s}(F)=1$ if $s\neq \beta$.
We split into two cases, $\beta=0$ and $\beta \geq 1$. Notice that the second can only happen when 
$p=N_{q^{\beta}}(F)\equiv F(1)^{\phi(q^{\beta})} \equiv  1$ mod $q$.

\vspace{1ex}
\noindent
{\bf Case (i)}: We have $N_{pq}(F)=1$, $N_{p}(F)=p$, $N_q(F)=1$.

 If $N_{pq}(F)=1$ then  $F(\omega_{pq})$ is a unit and by the lemma $\omega_{pq}^d F(\omega_{pq}^{-1})= \ve \omega_{pq}^J F(\omega_{pq})$ with $\ve=1$ or $-1$ and some  $J\geq 0.$ Hence we have a polynomial expression
$$
x^dF(x^{-1}) = \ve x^J F(x)+ \Phi_{pq}(x)h(x), $$
for some  $h(x)$ in $\mathbb Z[x].$

Suppose that $N_p(F)=p^{\delta}$ then $F(\omega_p)=(1-\omega_p)^{\delta}v(\omega_p)$, where $v(\omega_p)$ is a unit
in $\mathbb Z[\omega_p]$ and so $\omega_p^dF(\omega_p^{-1})=(-1)^{\delta} (1-\omega_p)^{\delta} \omega_p^{J'} v(\omega_p)=(-1)^\delta \omega_p^{J'}F(\omega_p)$.

Observing that $\Phi_{pq}(\omega_p)=q\Phi_p(\omega_p)^{-1}$ we get
$$ (-1)^{p\delta} F(\omega_p)^p \equiv \ve^p F(\omega_p)^p \text{ mod }q. $$
Since $|F(\omega_p)|_q=1$ we deduce that $(-1)^{\delta}=\ve$. Reversing the primes we see that $N_{pq}(F)=1$, $N_p(F)=p^{\delta}$, $N_q(F)=q^{\delta'}$ forces $(-1)^{\delta}=(-1)^{\delta'}=\ve$ ruling out $\delta,\delta'=1,0$ or $0,1$.
 
\vspace{1ex}
\noindent
{\bf Case (ii).} We have $N_{pq^{\beta}}(F)=p$ some $\beta\geq 1$ and $N_p(F)=1$.

From  $N_{pq^{\beta}}(F)=p$ we can write
$$ V:=\text{Norm}_{\mathbb Q(\omega_{pq^{\beta}})/\mathbb Q(\omega_p)} F(\omega_{pq^{\beta}})=\prod_{j\in \mathscr{J}} F(\omega_{pq^{\beta}}^j), \;\;\; \text{Norm}_{\mathbb Q(\omega_p)/\mathbb Q }(V) =p, $$
where we can take $\mathscr{J}$ to be the $\phi(q^{\beta})$ values of  $j$ mod $pq^{\beta}$ with $j\equiv 1$ mod $p$, $q\nmid j$. 
Hence $V=(1-\omega_p)v(\omega_p)$  where $v(\omega_p)$ is a unit in $\mathbb Z[\omega_p],$ and by Lemma \ref{Kron}
$$\prod_{j\in \mathscr{J}} F(\omega_{pq^{\beta}}^{-j})=(1-\omega_p^{-1})v(\omega_p^{-1})=-\omega_p^J V$$
for some $J,$ giving us the polynomial relationship
\be \label{Npq}  x^{pd\sum_{j\in\mathscr{J}}j}\prod_{j\in \mathscr{J}} F(x^j)^p+  \prod_{j\in \mathscr{J}}x^{pjd} F(x^{-j})^p=\Phi_{pq^{\beta}}(x)h(x), \ee
for some $h(x)$ in $\mathbb Z[x]$.
Now if $N_p(F)=1$ then  $\prod_{j\in \mathscr{J}} F(\omega_p^j)$ is a unit in $\mathbb Z [\omega_p]$ and Lemma \ref{Kron} gives    $\prod_{j\in \mathscr{J}} F(\omega_p^{-j})=\omega_p^{J'} \prod_{j\in \mathscr{J}} F(\omega_p^{j})$ for some $J'$. But then from  \eqref{Npq} 
$$\Phi_{p q^{\beta}}(x)=\Phi_{q^{\beta}}(x^{p})\Phi_{q^{\beta}}(x)^{-1} \;\; \Rightarrow \;\;   2\prod_{j\in \mathscr{J}} F(\omega_p^j)^p\equiv \; 0 \; (\text{mod }q), $$
a contradiction as the unit  $\left| \prod_{j\in \mathscr{J} } F(\omega_p^j)^p\right|_q=1$. \end{proof}

For small $n=pq$ using SAGE to work out an explicit set of $k= \frac{1}{2}\phi(n)-1$ generating units $u_1(\omega_{n}),\ldots ,u_{k}(\omega_n),$  in $\mathbb Z [\omega_{n}],$ it was noticeable that we could take $u_1(x)=x-1$ and 
all the others reciprocal  (making Lemma \ref{Kron}  self evident in those cases):

\vspace{2ex}
\begin{center}
\begin{tabular}{|c|l|}
\hline
$n$ & $u_2(x),\ldots ,u_k(x)$  \\
\hline
15 & $x+1,x^3+1$ \\
21 &  $x+1,x^2+1,x^3+1,x^6+x^3+1$   \\
33 & $x+1,x^2+1, x^3+1, x^4+1, x^6+1, x^{18}+1, x^6+x^3+1, \Phi_5(x)$ \\
35 &   $x+1,x^2+1,x^3+1,x^4+1,x^5+1,x^7+1,x^{15}+1,x^2+x+1,x^6+x^3+1,\Phi_{11}(x)$ \\
39 &   $x+1,x^2+1,x^3+1,x^5+1,x^6+1,x^{18}+1,x^6+x^3+1,\Phi_5(x),\Phi_7(x), \Phi_{11}(x)$ \\
\hline
\end{tabular}
\end{center}

\vspace{2ex}
\noindent
For 33 and 39  we have replaced the non-reciprocal unit given by SAGE by a reciprocal unit;
$(x^{11}-x^4+1)-(1+x^{11}+x^{22})=-x^4(x^{18}+1)$, $(x^{15}+x^2+1)+(1+x^{13}+x^{26})(x^{13}-x^2-1)=x^{28}(x^{11}-1)$ respectively, and for $33$ observed that $(x^{19}-x^{18}-x^{17}+x^{16}-x^{14}+x^{13}-x^{11}+x^{10}-x^6+x^4-x^3+x-1)+\Phi_{33}(x)-x^7(x^2-1)(x^{22}+x^{11}+1)=-x^{29}(x^2-1)$. For small $n=p$
we could make them all reciprocal:
\vspace{2ex}
\begin{center}
\begin{tabular}{|c|l|}
\hline
$n$ & $u_2(x),\ldots ,u_k(x)$  \\
\hline
5 & $x+1$ \\
7 & $x+1,x^3+1$\\
11 & $x+1,x^2+1,x^5+1, x^2+x+1$\\
13 & $x+1,x^2+1, x^6+1, x^2+x+1,x^{10}+x^5+1$ \\
\hline
\end{tabular}
\end{center}

\noindent
It is tempting to ask:

\vspace{1ex}
\noindent
{\bf Question 2.} Is there always a set of unit generators for $\mathbb Z[\omega_n]$ with at most one skew-reciprocal and the rest reciprocal?

\vspace{1ex}
In the proof of Theorem \ref{generalisation} for $n=pq$ we needed to separately consider the possibility of $N_{pq}(F)=p$
with $N_{p}(F)=1$.
It seems natural to ask:

\vspace{1ex}
\noindent
{\bf Question 3.} For which $p\equiv 1$ mod $q$ is there an $F$ in $\mathbb Z[x]$ with $N_{pq}(F)=p?$

\vspace{1ex}
\noindent
Some cases can be immediately dismissed:
\begin{lemma} \label{Nop}  If $pq\equiv 3$ mod 4, then there is no $F$ in $\mathbb Z[x]$ with $N_{pq}(F)=p$.

\end{lemma}
\begin{proof} Set $L=\mathbb Q(\sqrt{-pq})\subset \mathbb Q(\omega_{pq})$. If $p=N_{pq}(F)$ then $p=\text{Norm}_{L/\mathbb Q}( W),$  $W=\text{Norm}_{\mathbb Q(\omega_{pq})/L}(F(\omega_{pq}))$. But $p$ is not the norm of an algebraic integer in $L,$ since $x^2+pq y^2=4p$ plainly has no integer solution. \end{proof}

\noindent
If $p,q\equiv 3$ mod 4, $p\equiv 1$ mod $q$ then such an $F$ can exist; for example in $\mathbb Z[\omega_{21}]$ all ideals are principal and $N_{21}(x^4+x-1)=7$, though this is the only example that we found. For $p=19,31,43,67$ or $79$  there is no $F$ in $\mathbb Z[x]$ with $N_{3p}(F)=p$. For $p=19$ or 31 one can check in Magma  that the ideals $<p,x-\alpha_i>,$ where
$\Phi_{pq}(x)\equiv \prod_{i=1}^{q-1}(x-\alpha_i)^{(p-1)}$ (mod $p$), are non-principal,  though there is an $F$ in $\mathbb Q[x]$ for both. For the remaining primes one can check that there is no algebraic integer of norm $p$ in the 
degree $2(p-1)/6$ subfield $\mathbb Q \left( \sqrt{3}i \sum_{j=1}^6 \omega_p^{r^{(p-1)j/6}}\right)$, $r$ a primitive root mod $p$, though
again there are elements of norm $p$ in the field. It is easier to check that there are no integer solutions among the 19 values $pq<5000$ with $q>3$; using  the degree $(q-1)$ subfield $\mathbb Q\left( 2\cos(2\pi/q)\sqrt{p}i\right)$ for $q=7,11,19$ or the degree four field $\mathbb Q\left(\sqrt{p}i,\sqrt{q}i\right)$ for $q=23$.

\vspace{1ex}
There are no cases of $N_{n}(F)=p$ with $n=pq<10000$ and $p\equiv 1$ (mod 4), $q\equiv 1$ (mod $4q$). Of the 44 possibilities, 37  could be ruled
out just by checking that $p$ was not the norm of an algebraic integer in the  quadratic field $\mathbb Q(\sqrt{n})$
(note this will not rule out cases such as  $p=4q+1$, $q=13,37,53,73,\ldots $, where $p^2-n\cdot 2^2=p$), 
the other 7 using a degree 8 field and Magma; $\mathbb Q(\omega_5\sqrt{p}),$ when $q=5$ and $p=181,$ 761,1021,1621,1741, 
and  $\mathbb Q \left( \left(\sum_{j=1}^{(q-1)/4} \omega_q^{2^{4j}} \right)\sqrt{p}\right)$ when $(q,p)=(13,53),(37,149)$.


Rachel Newton has pointed out to us that \cite{Wei}  could probably  be used to decide when there are solutions to $N_{pq}(F)=p$ with $F$ in $\mathbb Q[x]$. See also \cite{Simon}.

\section{Constructing the multiples of $p^3$ and $3^3$ in $S_{3p}$} \label{p^3}

We can get $p^3m$ for any $3\nmid m$ as a $\mathbb Z_{3p}$ determinant from:

\begin{lemma} \label{ppower}
Suppose that $2mp\equiv k$ mod 3 with $k=1$ or $2$, then
$$ F(x)=\left(\frac{x^{mp+2k}-1}{x-1}\right) -x^{2mp-k}\left(\frac{x^k-1}{x-1} \right)(x^{3k}+1) \hspace{3ex}  \Rightarrow  \hspace{3ex} M_{3p}(F)=p^3m. $$

\end{lemma}

We can get any $3^3m$, $p\nmid m$ as a $\mathbb Z_{3p}$ determinant from:

\begin{lemma} \label{3power}
If $p\nmid m$ then
$$ F_3(x)=\left(\frac{x^{3p-9}-1}{x-1}\right) -x^{3p-6}(1+x^3+x^6)\left(\frac{x^{p-3-3m}-1}{x-1}\right) \;\;\; \Rightarrow \;\;\; M_{3p}(F_3)=3^4m, $$
multiplying by a power of $x$ to make a polynomial if $3m>p-3$.

For $\gcd(m,3p)=1$ we trivially  have 
$$ F_4(x)=(1+x^3+x^6) \prod_{q^{\alpha}\parallel m} \Phi_q^{\alpha}(x) \hspace{3ex} \Rightarrow \hspace{3ex} M_{3p}(F_4)=3^3m. $$
\end{lemma}
\begin{proof}[Proof of Lemma \ref{ppower}]
 The value at $x=1$ gives $N_1(F)=(mp+2k)-2k=mp$. 
When $x$ is a primitive cube root of unity we get $F(x)=\left(\frac{x^k-1}{x-1}\right) (1-2)$ and $N_3(F)=1$.

When $x$ is a primitive $p$th root of unity we get
$$ F(x)=\left(\frac{x^{2k}-1}{x-1}\right) \left( 1 - x^{-k}\left( x^{2k}-x^k+1\right)\right)=-x^{-k}(x^{2k}-1)\frac{(x^k-1)^2}{(x-1)}$$ 
and $N_p(F)=p^2$. 

When $x$ is a primitive $3p$th root, we have $x^{2m'p}+x^{m'p}+1=0$ for $3\nmid m'$, and  with $m'=m,2m$, 
\begin{align*}  (x-1)F(x) & = x^{mp+2k}-1-x^{2mp+3k}-x^{2mp}+x^{2mp-k}+x^{2mp+2k} \\
 & = x^{4mp} -x^{2k} -x^{2mp+3k} +x^{2mp-k} =x^{2k} (x^{2mp-3k}-1)(x^{2mp+k}+1)
\end{align*}
giving a contributiom $N_{3p}(F)=1$ (since the resultant of $\Phi_{3p}$ and $\Phi_d$ will be trivial unless $d$ is a prime power multiple of $3p$ or $3$ or $p$).

\end{proof}
 
\begin{proof}[Proof of Lemma \ref{3power}] Plainly $x=1$ contributes $N_1(F_3)=(3p-9)-3(p-3-3m)=9m$. For $x$ a $3p$th root of unity $x\neq 1$
we have:
\begin{align*} x^9 F_3(x) & =-\frac{1}{x-1} \left( x^9-1 + x^3(1+x^3+x^6)(x^{p-3-3m}-1)\right)\\
 & = -x^{-9}(1+x^3+x^6)\frac{(x^{p-3m}-1)}{(x-1)}. 
\end{align*}
For the primitive cube-roots this gives $-3\Phi_p(x)$ and $N_3(F_3)=9$. For the  primitive $p$th and $3p$th  roots we get 
plus or minus a power of $x$, $\Phi_9(x)$ and $\Phi_d(x)$ with $d>1$ dividing $\abs{p-3m}.$ None of these 
cyclotomics having order differing from $p$ or $3p$
by a prime power, so contribute $N_p(F_3)=N_{3p}(F_3)=1$.

For $F_4$, plainly $N_{3p}(1+x^3+x^6)=27,$ with the remaining factor contributing the $m$ as usual when $\gcd(|G|,m)=1$.

\end{proof}

For $n=35$ or $55$ we can get $p^3$ from
$$ M_{35}(1+x^3+x^5+x^7+x^{10})=M_{55}(1+x^3+x^5+x^7+x^{10})=5^3 $$
and
$$ M_{35}\left(  \left(\frac{x^{9}-1}{x-1}\right) -x^{3} (x^2+1) \right)=7^3,\;\; \;M_{55}\left(  x^{14}+1 +x^3\left(\frac{x^9-1}{x-1}\right)  \right)=11^3. $$
For $\gcd(k,n)=1$ and $F(1)\neq 0$ we have 
$$ G(x)=  \left(\frac{x^k-1}{x-1}\right)F(x)+\lambda \left(\frac{x^n-1}{x-1}\right) \Rightarrow M_n(G)= \left(\frac{kF(1)+\lambda n}{F(1)} \right) M_n(F), $$
and, since $F(1)=p$, taking $k=mp^t-q$, $\lambda=1$, in the above examples gives $M_n(G)=mp^{t+3}$  for any $t\geq 1$ 
and $q\nmid m$, with the $mp^3$ following from $k=m$, $\lambda=0$ or \eqref{basic} and closure under multiplication.

\section{Good or Bad 15-norms}

We begin by showing that elements in $\mathbb Z[\omega_{15}]$ can be written in one of two ways.

\begin{lemma} \label{goodbad} 
Suppose that $\xi$ is in $\mathbb Z[\omega_{15}]$ with $\gcd(N_{15}(\xi),15)=1$, then 
$$ \xi=u_1 F_1(\omega_{15})=u_2F_2(\omega_{15}) $$
where $u_1$,$u_2$ are units in $\mathbb Z[\omega_{15}]$ and
\begin{align*}  F_1(x) & =(x^5-1)\pm x^j \Phi_3(x) B(x) +(x-1)\Phi_3(x)\Phi_5(x)g_1(x),\\
   F_2(x) & =(x^3-1)\pm x^{j'}\Phi_5(x)B(x) +(x-1)\Phi_3(x)\Phi_5(x)g_2(x), \end{align*}
for some $g_1(x),g_2(x)$ in $\mathbb Z[x]$, integers $0\leq j,j'<15$, and either $B(x)=1$ or $(x-1)$.
\end{lemma}

We shall say that $\xi$ is {\em good} if $B(x)=1$ and {\em bad } if $B(x)=(x-1)$.

\begin{proof}

Suppose that $k=N_{15}(F)$ with $\gcd(k,15)=1$, then from 
$$ 1=-x\Phi_5(x) + (x^3+1)\Phi_3(x) $$
we can write
$$ F(x)= \alpha(x) (x^5-1) +  \beta(x) \Phi_3(x) - F(1)\Phi_5(x)\Phi_{15}(x), $$
with
$$ \alpha(x)= \frac{ (F(1)\Phi_{15}(x)- xF(x))}{x-1},\;\; \beta(x)= (x^3+1)F(x) $$
Dividing through by $\Phi_3(x)$ and reducing the coefficients of the remainder mod 5
$$ \alpha(x)= A(x) + 5 t_1(x) + q_1(x)\Phi_3(x), \;\;A(x)=Ax+B,  A,B\in \{0,\pm 1,\pm 2\}. $$
We can rule out $A=B=0$ since $5\nmid k$.
For the case $A$ or $B=0$ or $A=\pm B\neq 0$ we note that 
$$2x=(x-1)^2-\Phi_3(x)+5x,\;\;\; (x+1)=\Phi_3(x)-x^2, $$
and for the non zero $A$,$B$ with $A\neq \pm B$ that
\begin{align*} 2x+1=\Phi_3(x)-x(x-1),   &\;\;\;    x+2=(2-x)\Phi_3(x)+x^2(x-1),\\
x-2=2x(x-1)-2\Phi_3(x)+5x, & \;\;\; 2x-1=(4-2x)\Phi_3(x)+2x^2(x-1)-5. 
\end{align*}
Hence we can adjust $q_1(x)$  and  $t_1(x)$ to replace $A(x)$ by 
$$ A(x)=\pm x^j (x-1)^i,\;\;\; 0\leq i\leq 3. $$
Here we are only considering $F(x)$ on the 15th roots of unity so we are allowed to replace $x^{j}$ with $x^{j \text{ mod } 15}$ if $j$ is negative.

Replacing $\beta (x)$ by $(x-1)\beta_1(x) + \beta(1)\Phi_{15}(x)$ and repeating as necessary we
can write
$$ \beta(x)\Phi_3(x) =(x-1)^i \beta_2(x) \Phi_3(x)+ s_1(x)\Phi_3(x)\Phi_{15}(x). $$
Dividing by $\Phi_5(x)$ and reducing mod $3$
$$ \beta_2(x)=B(x)+ 3t_2(x) +q_2(x)\Phi_5(x), \;\; B(x)= a_3x^3+a_2x^2+a_1x+a_0, \;\; a_{\ell}\in\{0,\pm 1\}. $$
we can not have all the $a_{\ell}=0$ since $3\nmid k$.
Now, we are allowed to adjust by $x^5-1$ (by altering $q_2(x)$),  so that introducing an appropriate   power of $x$ we can think of $B(x)$ as 5  coefficients
$a_0,...,a_4$ arranged cyclically, with at most 4 of them $\pm 1$. If we have 4 of them non-zero, then we will have at least two coefficients the same sign and by subtracting or adding a $\Phi_5(x)$ as the value is 1 or -1
and reducing mod $3$ we can reduce to at most $3$ non-zero terms. A single term corresponds to $\pm x^j$
and two non zero terms to $\pm x^j(x\pm 1)$ or $\pm x^j(x^2\pm 1)$. If we have three terms and all are the same
then adding or subtracting a $\Phi_5(x)$ reduces to two non-zero terms. If all are consecutive in the cycle
we get $\pm x^j(x^2-x+1)$ or $\pm x^j(1+x-x^2)$ or $\pm x^j(1-x-x^2)$. If non-consecutive then we have two consecutive with one gap either end
reducing to 
$\pm x^j(1+x-x^3)$, $\pm x^j(1-x+x^3)$ or $\pm x^j(1-x-x^3)$, where we can write
$1+x-x^3=\Phi_5(x)-3x^3-x^2(1-x+x^2)$.

Thus, multiplying through by a power of $x$, changing the $M_{15}(F)$ by at most a sign, we can assume that
\begin{align}  F(x) = (x-1)^i( (x^5-1)\pm x^j \Phi_3(x)B(x) ) &  + 5t_1(x)(x^5-1)+3t_2(x)\Phi_3(x) \nonumber \\ & + t_3(x)\Phi_3(x)\Phi_5(x)+t_4(x)\Phi_{15}(x) \end{align}
where $B(x)$ is either a Type 1:
\be \label{type1}  1, (x+1), (x^2+1),  (1-x+x^2), \ee
or a Type 2:
\be \label{type2}  (x-1), (x^2-1), (1+x-x^2),(1-x-x^2),(1-x-x^3),(1-x+x^3). \ee
So far we have preserved the values at all the 15th roots of unity; we will need this decomposition in the proof of Theorem \ref{Z15}. For the Lemma we just need to preserve the value at the primitive $15$th roots of unity, up to multiplication by a unit in $\mathbb Z [\omega_{15}]$.
The Type 1 will give us the good cases and the Type 2 the bad.

For a primitve $15$th root of unity $x$ we have 
$$ (x-1)^{-1} = (x^2-1)(x^4-1)(x^7-1)(x^8-1)(x^{11}-1)(x^{13}-1)(x^{14}-1) $$
and
$$ 3=(x^5-1)(x^{10}-1), 5= (x^3-1)(x^6-1)(x^9-1)(x^{12}-1), $$
hence we can divide through by $(x-1)^i$ and up to a unit replace $F$ by
$$ F(x)=(x^5-1)\pm x^j \Phi_3(x)B(x) + g_3(x)\Phi_3(x)\Phi_5(x). $$
When $B(x)=(x+1)$ or $(x^2+1)$ multiplying through by the other gives $(x^2+1)(x+1)=-x^4+\Phi_5(x)$
and when $B(x)=1-x+x^2$ multiplying by $(x+1)^2$ gives $(x^3+1)(x+1)=-x^2+\Phi_5(x)$ reducing the Type 1 to $B(x)=1$, where since $(1+x),(1+x^2)$ are $-x^2$ or $-x$ mod $\Phi_3$ this just changes $(x^5-1)$
by a power of $x$ which can be divided out. 
Similarly
$$ (1+x-x^2)(1+x^2)=x^2(x^2-1)+\Phi_5(x)-3x^4,   (1-x-x^2)(1+x^2)=(x^2-1)-\Phi_5(x)+3, $$
$$ (1-x-x^3)(1+x)=(x-1)-\Phi_5(x)+3,\;\; (1-x+x^3)(1+x)=x(x-1)+\Phi_5(x)-3x^2$$
with multiplication by $(x^2+1)$ removing the $(x+1)$, reducing $(x^2-1)$ to $(x-1)$.
Hence in Type 2 we can always reduce to $B(x)=(x-1)$.
Thus for $x$  a 15th root of unity we have shown
$$ F(x)=u\left( (x^5-1)\pm x^j \Phi_3(x)B(x) + g_3(x)\Phi_3(x)\Phi_5(x) \right) $$
with $u$ a unit in $\mathbb Z[x]$ and  $B(x)=1$, which we will call a good case (from Type 1), or $B(x)=(x-1)$, which we will call bad from the Type 2.
Notice we can always replace $g_3(x)$ by $g_3(x)-g_3(1)\Phi_{15}(x)$ and hence the $g_3(x)\Phi_3(x)\Phi_5(x)$
by $g_4(x)(x-1)\Phi_3(x)\Phi_5(x)$ as stated in the Lemma.

When $B(x)=(x-1)$ we plainly obtain the second form with $j'=15-j$ by dividing by $\pm x^j$. When $B(x)=1$ we 
multiply by $x^4(x+1)$ where
for the first term
$ x^4(x+1)=-x^6 +x^4\Phi_3(x)$
and for the second
$$ x^4(x+1) = (x+1)\Phi_{15}(x) -(x-1)^2(x^4-\Phi_5(x))(x^2-\Phi_5(x)) $$
hence we get
$$-x^6 (  (x^5-1) \pm x^j (x^3-1)(x-1) )+ (x-1)\Phi_3(x)\Phi_5(x)g_2(x),$$
and dividing out an $(x-1)$ and power of $x$ gives the representation that we want with $ j'=15-j.$
Likewise we can get the first from the second.

\end{proof}

The concept of good or bad is well defined and we have a parity type relationship.

\begin{lemma} An element $\xi$ in $\mathbb Z[\omega_{15}]$, $\gcd(N_{15}(\xi),15)=1$,   can not be both good and bad,
and $1$ is bad. If $\xi_1$, $\xi_2$ are both good or both bad then $\xi_1\xi_2$ is bad, otherwise $\xi_1\xi_2$ is good.
The conjugates of $\xi$ are either all good or all bad.
\end{lemma}

\begin{proof}
We rule out $F(x)$ being both good and bad:
$$ u_1\left((x^5-1)\pm x^j\Phi_3(x)+ g_4(x)\Phi_3(x)\Phi_5(x)\right)=u_2\left(\Phi_5(x)\pm x^{j'}\Phi_3(x)+ g_5(x)\Phi_3(x)\Phi_5(x)\right). $$
Taking a basis $(x-1),(x+1),(x^3+1)$ for the units,  moving any negative powers to the other side, we obtain equality
at the 15th roots of unity and hence  a  polynomial 
identity:
\begin{align*}\pm x^k  (x-1)^{r_1}(x+1)^{s_1}(1+x^3)^{t_1}\left( (x^5-1)\pm x^j\Phi_3(x) \right) & =  \\ (x-1)^{r_2}(x+1)^{s_2}(1+x^3)^{t_2}\left( \Phi_5(x) \pm x^{j'}\Phi_3(x) \right)  + & g(x)\Phi_3(x)\Phi_5(x) + h(x)\Phi_{15}(x). 
\end{align*}
Observing that $\Phi_{15}(\omega_3)=-5\omega_3$, $\Phi_{15}(\omega_5)=3(\omega_5^3+1),$  taking $N_3$ and $N_5$ gives
$$ 3^{r_1+1}4^{t_1}\equiv 3^{r_2} 4^{t_2}\text{ mod 5, } \;\;\; 5^{r_1}\equiv 5^{r_2} \text{ mod 3}, $$
but the first requires $r_1$,$r_2$ to have opposite parity and the second the same parity. Since $(x^5-1)-(x^3-1)$ is a unit
we see that $1$ (or any unit) must be bad.

Observe that the product of a bad and a good case gives a good case and the product of two good or two bad cases a bad case; essentially this is the parity of the power of $(x-1)$ on the $\Phi_3(x)$ and $\Phi_5(x)$ terms, but to be explicit
the product leads to
$$ u\left( (x^5-1)(x-1)\Phi_5(x)\pm x^j \Phi_3 (x) \Phi_3(x)B_1(x)B_2(x)\right)  \text{ mod } \Phi_3(x)\Phi_5(x). $$
If one is bad, say $B_1(x)=(x-1)$, then we can write $\Phi_5(x)=\Phi_3(x)+x^3(x+1)$ and factor out the $x^3(x+1)(x-1)$
to get $u((x^5-1)\mp x^j \Phi_3(x)B_2(x))$ and we get good or bad as the other is good or bad. 
In the case both are good $B_1(x)=B_2(x)=1$ we write $\Phi_5(x)=-x^5+ (1+x^3)\Phi_3(x)$ and $\Phi_3(x)=(x-1)^2+3x$
and factoring out $-x^5(x-1)$ produces a bad case.

Notice also if $F(x)$ gives a good or bad case then so do all its conjugates $F(x^k)$, $\gcd(k,15)=1$.
To see this observe that 
$(x^{5k}-1)=(x^5-1)t(x)$ where $$ t(x)=(1+x^5+\cdots + x^{5(k-1)}) \equiv 1 \text{ or } -x^{10} \text{ mod $\Phi_3(x)$  as $k\equiv 1$ or 2 mod 3}. $$
When $\Phi_3(x)B(x)=(x^3-1)$ we use
$ (x^{3k}-1)=(x^3-1)t_2(x)$ where $t_2(x)=1+x^3+\cdots +x^{3(k-1)} $ mod $\Phi_5(x)$ takes the
form $1$, $1+x^3$, $1+x^3+x^6=-x^2(1+x^2)$, $-x^{12}$ as $k\equiv 1,2,3$ or $4$ mod 5, which can all be removed
to return to a bad case. In the good case we multiply by the unit $(x^k-1)$ where $1+x+\cdots + x^{k-1}\equiv 1$ or $-x^2$ mod $\Phi_3(x)$ enabling $(x-1)$ to be factored out to return to a good form.

\end{proof}

Recalling that $\mathbb Z[\omega_{15}]$ has uniqueness of factorisation we can see that we can divide the primes 
in $\mathbb Z[\omega_{15}]$, not dividing 15,  into those with good or bad representations and a  product of these will be good iff 
if it contains an odd number of good primes (unchanged by which conjugate is used).
In particular for an integer $k$ which is a 15-norm, $k=N_{15}(\xi)$, $\gcd(k,15)=1$,  it makes sense to say that $k$ is {\em 15-norm good}
or {\em bad} as $\xi $ is good  or bad. In particular a $k$ will be  norm-15 good iff its factorisation in $\mathbb Z$  contains an odd number of  norm-15 good  prime powers $p^r=N_{15}(\mathscr{P})$, 
where ${\mathscr{P}}$ is a good prime in $\mathbb Z[\omega_{15}]$.

Finally we can obtain the multiples of any integer which is 15-norm good:

\begin{lemma}\label{goodachieve} We can obtain $3^2k$ and $5^2k$ as a $15\times 15$ circulant determinant  for any  $k$, $\gcd(k,15)=1$,  which is  15-norm good. \end{lemma}

\begin{proof}
If $k$ is a good 15-norm then $k=N_{15}(F)$ for an $F$ of the form 
$$F(x)=(x^5-1)\pm x^j \Phi_3(x)+(x-1)\Phi_3(x)\Phi_5(x)g_1(x)$$
and $M_{15}(F)=\pm 3^2k$ with $N_1(F)=\pm 3$, $N_3(F)=3$ and $N_5(F)=1$.

Likewise we get $\pm 5^2 k$ from  $F(x)=(x^3-1)\pm x^j \Phi_5(x)+(x-1)\Phi_3(x)\Phi_5(x)g_2(x)$. 

\end{proof}

\section{Proof of Theorem \ref{Z15}} \label{thm}

In this section we prove the following:

\begin{theorem} If $F$ is in $\mathbb Z [x]$ and $M_{15}(F)=3^2m$ or $5^2m$ with $\gcd(15,m)=1$, then either $N_{15}(F)$ is 15-norm good, or $p\mid N_3(F)$ for
some prime $p\equiv 7$ or 13 mod 15, or $p\mid N_5(F)$ for some prime $p\equiv 11$ mod 15.
\end{theorem}

From Lemma \ref{goodachieve} we know that we can achieve $k3^2$ and $k5^2$ when $k$ is 15-norm good
and in the next section we show that we can achieve $3^2p$ and $5^2p$ whenever $p\equiv 7,11$ or 13 
mod 15. We can achieve any  multiple $m$ of these with 
$\gcd(m,15)=1$. So these are exactly the $15\times 15$ determinant values.

To complete the proof of Theorem \ref{Z15} we observe that a 15-norm good $k$  must contain at least one norm-15 good  prime power $p^r=N_{15}(\mathscr{P})$ 
where ${\mathscr{P}}$ is a good prime in $\mathbb Z[\omega_{15}]$. We characterize these $p^r$ in Section \ref{factorsinZ15}.

\begin{proof}

Suppose $F\in \mathbb Z[x]$ has $M_{15}(F)=3^2k$ or $5^2k$, $\gcd(k,15)=1$.
Since $N_3(F)\equiv F(1)^2$ mod 3, $N_{15}(F)\equiv N_5(F)^2$ mod 3, $N_5(F) \equiv F(1)^4$ mod 5, $N_{15}(F)\equiv N_3(F)^4$ mod 5,
and since 3 and 5 remain irreducible in $\mathbb Z[\omega_5]$ and $\mathbb Z[\omega_3]$ respectively we can't have
$3\parallel N_5(F)$ or $5\parallel N_3(F)$, we must have $\gcd(15,N_{15}(F))=1$ and $3\nmid N_5(F)$, $5\nmid N_3(F)$
and in the first case $3\parallel F(1)$, $N_3(F)=3m_1$, $N_5(F)=m_2,$ $\gcd(m_1m_2,15)=1$ and in the 
second  $5\parallel F(1)$, $N_3(F)=m_1$, $N_5(F)=5m_2,$ $\gcd(m_1m_2,15)=1$.

We suppose we have a  Type 2 decomposition \eqref{type2}. Observe that all the $B(x)$ of Type 2 have $N_5(B(x))=5$ or $11$.
Since  $\Phi_{15}(\omega_5)=3(1+\omega_5^3)$, $\Phi_{15}(\omega_3)=-5\omega_3$ we obtain
$$ N_5(F) \equiv 5^i\cdot 2 \equiv (-1)^{i+1} \text{ mod 3 },\;\;\; N_3(F)\equiv 3^{i+1} \text{ mod } 5. $$

Suppose that $i$ is even, then in the first case $m_2\equiv -1$ mod $3$ so the factorisation of $m_2$ must contain an odd 
power of a prime $p\equiv 2$ mod 3. Since the power is odd the prime must split completely in $\mathbb Z[\omega_5]$
so must be 1 mod 5. That is $N_5(F)$ must contain a prime $p\equiv 11$ mod 15.  In the second case we get $m_1\equiv \pm 2 $ mod 5. Hence $m_1$ must contain an odd power
of a prime $p\equiv \pm 2$ mod 5 and since it's an odd power in $N_3(F)$ must be 1 mod 3. That is $p\equiv 7$ or 13 
mod 15.

If $i$ is odd then in the first case $m_1\equiv \pm 2$ mod 5 and in the second $m_2\equiv -1$ mod 3 with the same conclusions.

If we have a Type 1 decomposition then, as in the proof of Lemma \ref{goodbad},  $N_{15}(F)$ is 15-norm good.

\end{proof}

\section{Obtaining $3^2p$ and  $5^2p$ for $p\equiv 7,11$ or $13$  mod 15.}

To show that we obtain $3^2p$ and $5^2p$ for all the $p\equiv 7$ or 13 mod 15 we begin by showing that these primes
must be 3-norms of a particular form:

\begin{lemma}\label{L7mod15} If $p\equiv 1$ mod 3 then 
$$p=N_3(a+bx +5(Ax+B))$$
 for some $A,B\in \mathbb Z,$ with $(a,b)=(1,0),(2,0),(3,1),(4,3)$ as $p\equiv 1,4,7,13$ mod 15.

If $p\equiv 7$ or 13 mod 15 then
$$ p=N_3( a+bx + 5(x-1)(Cx+D)) $$
for some $C,D\in \mathbb Z$ with  $(a,b)=(2,3)$ or $(3,-1)$   as $p\equiv 7$ or 13 mod 15.

\end{lemma}

\begin{proof}[Proof of  Lemma \ref{L7mod15}]
Since $p\equiv 1$ mod 3, we know that $p$ splits in $\mathbb Z[\omega_3]$ and 
$$p=N_3(\alpha+\beta x )=\alpha^2+\beta^2-\alpha\beta,$$
for some integers $\alpha$, $\beta$, and we can write $\alpha+\beta x=a+bx + 5(Ax+B)$ with $a,b\in \{0,\pm 1,\pm 2\}$,
not both zero since $5^2\nmid p$.
Observe that the norm remains unchanged if we switch the positions of $\alpha$ and $\beta$ or multiply by $-1$;
that is we can replace $(a,b)$ by $(b,a)$ or $(-a,-b)$ on  replacing $Ax+B$ by $Bx+A$ or $-Ax-B$. 
Hence we can assume that $a=1$ or 2 and $|b|\leq a$.

 If $b=0$ that gives us $(2,0)$ or $(1,0)$.
We can also replace $\alpha+\beta \omega_3$ by its conjugate $\alpha+\beta \omega_3^2$, and hence  $(a,b)$ by $(a-b,-b),$
on replacing $Ax+B$ by $-Ax+(B-A)$. Hence $(a,a)$ reduces to $(0,-a)$ and thence to  $(1,0)$ or $(2,0)$. Similarly $(1,-1)$
reduces  to $(2,1)$ and $(2,-2) \mapsto (4,2) \mapsto (-1,2) \mapsto (2,-1)$.
This just leaves $(2,1)\mapsto (-3,-4)\mapsto (4,3)$ or $(2,-1)\mapsto (-3,-1)\mapsto (3,1)$. Since $p\equiv a^2-ab+b^2$ mod 5 these four types $(a,b)=(1,0),(2,0),(3,1),(4,3)$ correspond to the four possibilities mod $15$.

Notice that for $p=7$ or 13 mod 15 we could alternatively take $(2,-1)\mapsto (2,-2)\mapsto (2,3)$ or $(2,1)\mapsto (-3,1)\mapsto (3,-1)$ and write $p=N_3(c+dx+5(Ax+B))$ with $(c,d)=(2,3)$ or $(3,-1)$ respectively.

Now if $A+B=3m+r$, $r=0,\pm 1$  we have 
$$ Ax+B=A(x-1)-m(x-1)(x+2) +r \text{ mod }\Phi_3(x), $$ 
and we can write $p=N_3(c+dx+5r + (x-1)(Cx+D)). $
If $r=0$ we are done. We can also rule out $r=-1$ as $3\mid N_3(-3+3x), N_3(-2-x)$.
If $p=7$ mod 15 and $r=1$ we write 
\begin{align*} p & =N_3( 2+3x+ 5x^2+ 5(x-1)(Cx+D)-5(x^2-1))\\
& =N_3(-3-2x+5(x-1)(C_1x+D_1))=N_3(-x(-3-2x^2+5(x^2-1)(C_1x^2+D_1)))\\
& =N_3(2+3x+5(x-1)(C_2x+D_2)).
\end{align*}
If $p=13$ mod 15 and $r=1$
\begin{align*} p & =N_3( 3+4x+ 5(x-1)(Cx+D-1))= N_3(-x^2(3+4x^2+5(x^2-1)(Cx^2+D-1))\\ & =N_3(3-x+5(x-1)(C_1x+D_1)).
\end{align*}

\end{proof}

With the $A,B$ and $C,D$  from Lemma  \ref{L7mod15} we have

\begin{theorem}\label{7mod15}
If $p\equiv 7$ mod 15, then
\begin{align*} M_{15}\left(1-x+x^3\Phi_3(x^3)+(1-x)\Phi_5(x)\Phi_{15}(x)(Ax+B)\right)=3^2p,\\
 M_{15}\left(1-x^2+x^4+x^9+x^{10}+x^{13}+x^{14}+(x-1)\Phi_5(x)\Phi_{15}(x)(Cx+D)\right)=5^2p.
\end{align*}
If $p\equiv 13$ mod 15
\begin{align*} M_{15}\left(1-x^5-x^{11}+x^{12} +x^3\Phi_3(x^3)+(1-x)\Phi_5(x)\Phi_{15}(x)(Ax+B)\right)& =3^2p,\\
M_{15}\left(1+x^3+x^6+x^9+x^{14}+(x-1)\Phi_5(x)\Phi_{15}(x)(Cx+D)\right) & =5^2p.
\end{align*}
\end{theorem}

\begin{proof} Denote the first and second polynomials by $F$ and $G$. Observing that 
$$ N_1N_5N_{15}(1-x+x^3\Phi_3(x^3))=N_1N_5N_{15}(1-x^5-x^{11}+x^{12} +x^3\Phi_3(x^3))=3, $$
plainly  $N_1N_5N_{15}(F)=3$. We have $\Phi_5(\omega_3)\Phi_{15}(\omega_3)=5$, 
$$ 1-\omega_3+3= (1-\omega_3)(3+\omega_3),\;\; 1-\omega_3^5-\omega_3^{11} +\omega_3^{12} +3=(1-\omega_3)(4+3\omega_{3}) $$
and hence $N_3(F)=N_3(1-x)N_3(a+bx+5(Ax+B))=3p$.

Similarly 
$$ N_1N_5N_{15}\left( 1-x^2+x^4+x^9+x^{10}+x^{13}+x^{14}\right)=N_1N_5N_{15}\left(1+x^3+x^6+x^9+x^{14} \right)=5^2$$
and $N_1N_5N_{15}(G)=5^2$. Since
$$ 1-\omega_3^2+\omega_3^4+\omega_3^9+\omega_3^{10}+\omega_3^{13}+\omega_3^{14}=2+3\omega_3,
1+\omega_3^3+\omega_3^6+\omega_3^9+\omega_3^{14}=3-\omega_3, $$
we get $N_3(G)=N_3(c+dx+5(x-1)(Cx+D))=p$.

\end{proof}
To show that we can obtain all the $3^2p$ and $5^2p$ with  $p\equiv 11$ mod 15 we need a similar $5$-norm representation lemma:

\begin{lemma}
If $p\equiv 11$ mod 15 then
$$ p=N_5( 3\pm (x-1)+3(x-1)g(x)) $$
for some $g(x)$ in $\mathbb Z [x]$, and
$$ M_{15}\left(  1+x^5+x^{10} \pm (x-1)+ (1+x^5+x^{10})(x-1)g(x)\right)=3^2p. $$

We can also write
$$ 5p=N_5( (x-1)(1+2x)+3(x-1)g_2(x)) $$
for some $g_2(x)$ in $\mathbb Z [x]$, and
$$M_{15}\left(    x^{13}(1+x+x^2+x^3)-x^7+x^{10}  +x^{11} + (1+x^5+x^{10})(1-x)g_2(x)\right)= 5^2p. $$
\end{lemma}

\begin{proof} Since $p\equiv 1$ mod $5$ we know that $p$ splits completely in $\mathbb Z [\omega_5]$ and $p=N_5(F(x))$ for some $F\in \mathbb Z [x]$.
Replacing $F(x)$ by $\pm F(x),\pm (x+1)F(x)$  we can assume that $F(1)\equiv 3$ mod 5.  Hence we can write
\begin{align*} F(x) & =3+ 5m +(x-1)g_1(x) \\
& = 3+ (x-1)g_2(x),\;\; g_2(x)=g_1(x)+ m(x^2-1)(x^3-1)(x^4-1) \\
& =3+(x-1)h(x)+ 3(x-1)g_3(x),\;\; h(x)=\sum_{j=0}^3a_jx^j, \; a_j\in \{0,\pm 1\}. \end{align*}
Proceeding as in the proof of Lemma \ref{goodbad} we can reduce to $h(x)=\pm x^jB(x)$ with $B(x)$ of type \eqref{type1}
or type \eqref{type2}. But the type 2 have $N_5(B(x))=5$ or $11$, resulting in $N_5(F)\equiv 1$ mod 3.
So we can assume that
$$ p=N_5( 3\pm x^j(x-1)B(x)+3(x-1)g(x)), \;\; B(x)=1,(x+1),(x^2+1) \text{ or } x^2-x+1. $$
If $B(x)=x+1$ we can make the substitution $x\mapsto x^3$ to make the
$(x-1)B(x)=(x^2-1) \mapsto (x-1)$. Writing $(x^2-x+1)^{-1}=\prod_{j=2}^4 \left(x^{2j}-x^j+1\right)=1+(x-1)g_3(x)$ we can 
divide out  $B(x)=(x^2-x+1)$. If $B(x)=x^2+1$ then $x\mapsto x^3$ makes $(x-1)B(x) \mapsto (x^3-1)(x+1)=(x-1)B(x)$,
with $B(x)=(x+1)(x^2+x+1)$, but in this case $B(x)^{-1}=\prod_{j=2}^4 B(x^j),$ $B(x^j)=6+(x-1)g_1(x)=1+(x-1)g_2(x)$
and we can again divide by $B(x)$. This leaves $B(x)=1$ and we can divide out any  $x^j$ by multiplying through 
by $x^{4j}=1+(x-1)g_4(x)$.

For the $F$ given we have $N_1(F)=3$, $N_3N_{15}(F)=N_3N_{15}(\pm (x-1))=3$ while
$ N_5(F)=N_5(3\pm (x-1)+3(x-1)g(x))=p. $

We can write 
\begin{align*} 5p & =N_5((1-x)\left(3\pm (x-1) +3(x-1)g_1(x)\right)) \\ & =N_5( \pm (x-1) (1+2x)+3(x-1)\left((x-1)g_2(x)-1\mp x\right).
\end{align*}

For the $G(x)$ given we have
$$ N_1N_3N_{15}(G)=N_1N_3N_{15}( x^{13}(1+x+x^2+x^3)-x^7+x^{10}  +x^{11} )=5 $$
while
$$\omega_5^{13}(1+\omega_5+\omega_5^2+\omega_5^3)-\omega_5^7+\omega_5^{10}+\omega_5^{11}=(1-\omega_5)(1+2\omega_5)$$
and $N_5(G)=N_5\left((1-x)(1+2x)+3(1-x)g_2(x)\right)=5p$.

\end{proof}

\section{Primes in $\mathbb Z[\omega_{15}]$}\label{factorsinZ15}
Finally, to simplify the a statement of the Theorem, we need a lemma  to say how the primes split in $\mathbb Z[\omega_{15}]$:

\begin{lemma} The primes $p\neq 3,5$ factor in $\mathbb Z[\omega_{15}]$ as
\begin{align} \label{1}  p =1 \mod 15   \;\;  & \Rightarrow \;\;\; p=\mathscr{P}_1\cdots \mathscr{P}_8,\;\;\;  N_{15}(\mathscr{P}_i)=p,\\
p=4,11 \text{ or } 14  \mod 15  \;\; &\Rightarrow \;\;\; p=\mathscr{P}_1\cdots \mathscr{P}_4,\;\;\; N_{15}(\mathscr{P}_i)=p^2,\\
p=2,7,8 \text{ or } 13  \mod 15 \;\; &\Rightarrow \;\;\; p=\mathscr{P}_1\mathscr{P}_2,\;\;\; N_{15}(\mathscr{P}_i)=p^4.
\end{align}
For the remaining primes $5=u_1(1-\omega_5)^4$, $3=u_2(1-\omega_3)^2$ for some units $u_1,u_2$.
\end{lemma}

\begin{proof} Recall (eg Washington \cite[Theorem 2.13]{Washington}) that  $p\nmid n$ splits into $\phi(n)/f$  distinct primes in $\mathbb Q (\omega_n)$ each of which have residue class degree $f$, where $f$ is the smallest positive integer with $p^f\equiv 1$ mod $n$. When $n=15$ plainly $f=1$ if  $p\equiv 1$ mod 15, $f=2$ if $p\equiv 4,11$ or  $14$ mod 15 and $f=4$ if $p\equiv  2,7,8$ or 13 mod 15. Similarly $3$ and $5$ stay prime in $\mathbb Z[\omega_5]$ and $\mathbb Z[\omega_3]$ but ramify completely on adding $\omega_3$ or $\omega_5$. Since $\mathbb Q(\omega_{15})$ has class number one we can replace prime ideals with prime elements.
\end{proof}

Hence if $k$, $\gcd(k,15)=1$,  is a 15-norm then it consists of products of $p$ with $p\equiv 1$ mod $15$,  $p^2$ with  $p\equiv 4,11$ or 13 mod 15, and $p^4$ with $p\equiv \pm 2$ mod 5. A good 15-norm must be divisible by at least one of these $p$ or $p^2$ or $p^4$ which is  15-norm good.
For the $p\equiv 1$ mod 15 it is hard to predict which $p$ are good or bad, and for $p\equiv 7,11$ or 13 we can otherwise achieve all multiples $p$. For the remaining $p^2$, $p\equiv 4$ or 14 mod 15 and $p^4$ with $p\equiv 2$ or 8  mod   15
we can determine this:

\begin{lemma} \label{p^4}  $p^4$ is 15-norm good if $p\equiv \pm 2$ mod $5$.

\end{lemma}
Of course if $p\equiv \pm 1$ mod 5 then $p^2$ is a 15-norm and its square is 15-norm bad.

\begin{proof} Since $p\equiv \pm 2$ mod 5 we know that $p$ remains irreducible in $\mathbb Z[\omega_5]$. 

Suppose that $p^4=N_{15}(F(x))$ has a bad representation
$$ F(x)=(x^5-1)\pm x^j (x^3-1) + (x^3-1)(x^5-1) g(x). $$
Writing $G(x)=F(x\omega_3)F(x\omega_3^2)$ we have $G(x)\in \mathbb Z[x]$ with $p^4=N_5(G(x))$.
Hence for $x$ a primitive 5th root of unity we have $H(x)=G(x)G(x^{-1})=|F(x)|^2$  in $\mathbb Z\left[ \frac{1}{2}(1+\sqrt{5})\right],$
the integers in the  real subfield of the 5th roots of unity and 
$$ H(x)=u_1 p^2,   \;\;\;   H(x^2) = u_2 p^2, \;\;\; u_1= \left(\frac{1}{2}(1+\sqrt{5})\right)^{2k},\;\;  u_2= \left(\frac{1}{2}(1-\sqrt{5})\right)^{2k},$$
for some $k$ in $\mathbb Z$ (notice that $u_1$ and its conjugate $u_2$ under $x\rightarrow x^2$, $\sqrt{5}\rightarrow -\sqrt{5}$, must both be positive and so must be an even power $2k$ of the fundamental unit). 
\begin{align*} G(x)  = &\left( (\omega_3^2-1) \pm (x\omega_3)^j(x^3-1) +(x^3-1)(\omega_3^2-1)g(x\omega_3)\right)\\
 &\;\;\;  \left( (\omega_3-1) \pm (x\omega_3^2)^j(x^3-1) +(x^3-1)(\omega_3-1)g(x\omega_3^2)\right)\\
  =   & 3 + x^{2j} (x^3-1)^2 + 3(x^3-1)t_1(x), 
\end{align*}
and
$$ H(x)=G(x)G(x^{-1})= 9 +\frac{5}{2}(3+\sqrt{5}) + 3\sqrt{5}\: t_2\left( \frac{1}{2}(1+\sqrt{5})\right). $$
Writing
$ \left(\frac{3+\sqrt{5}}{2}\right)^k=\frac{1}{2}(a_k+b_k\sqrt{5})$ we have
$ a_kp^2 = H(x)+H(x^2) = 33+15m $
and, since $p^2\equiv -1$ mod $5$, we must have  $a_k\equiv 0$ mod 3 and $a_k\equiv 2$ mod $5$. But it is readily 
checked that the $(a_k \text{ mod }3, a_k \text{ mod } 5)$ cycle through the values $(0,-2),(1,2),(0,-2),(2,2)$ never $(0,2)$.
Hence $p^4$ must have a good representation.
\end{proof}

\begin{lemma} If $p\equiv 4$ mod 15 then $p^2$ is 15-norm good.

If $p\equiv 14$ mod 15 then $p^2$ is 15-norm bad. 

\end{lemma}

\begin{proof} We proceed as in the proof Lemma \ref{p^4}, except that when $p\equiv 4$ mod 5 we know that $p$ factors in the 
real subfield of the 5th roots of unity $p=\alpha^2-5\beta^2=\left(\alpha+\beta\sqrt{5}\right)\left(\alpha-\beta\sqrt{5}\right)$
(plainly $x^2+x-1$ factors mod $p$ since $\left(\frac{5}{p}\right)=\left(\frac{p}{5}\right)=1$), with
$\alpha \pm \beta \sqrt{5}$ remaining prime in $\mathbb Z[\omega_5]$ (fixed by $x\mapsto x^{-1}$ and interchanged by $x\mapsto x^2$).  Hence this time
$$ H(x)=\frac{1}{2}(a_k+b_k\sqrt{5})\left(\alpha+\beta \sqrt{5}\right)^2,\;\;\;  H(x^2)=\frac{1}{2}(a_k-b_k\sqrt{5})\left(\alpha-\beta \sqrt{5}\right)^2,$$
giving
$$ H(x)+H(x^2)=  (\alpha^2+5\beta^2) a_k + 10 \alpha\beta  \;b_k.     $$
Suppose that $p\equiv 4$ mod 15 and that we have a bad representation  then as before 
$$ H(x)+H(x^2) =33 +15m. $$
Since $p\equiv 1$ mod $3$ we know that $3\mid \alpha \beta$ and  $\alpha^2\equiv p\equiv -1$ mod 5.
Hence $-a_k\equiv -2$ mod $5$ and $0\equiv \pm a_k $ mod $3$, But as before $(a_k \text{ mod } 3, a_k \text{ mod } 5)$ is not $(2,0)$.
So the representation for $p^2$ must be good.

Suppose that $p\equiv 14$ mod 15 and that $p^2$ has a good representation:
$$ F(x)=(x^5-1)(x-1)\pm x^j (x^3-1) + (x^3-1)(x^5-1) g(x). $$
Then
$$ G(x)=3(1+x+x^2)+x^{2j}(x^3-1)^2 +3(x^3-1)t(x) $$
and
$$ H(x) = \frac{9}{2}(3+\sqrt{5})+\frac{5}{2}(3+\sqrt{5}) +3\sqrt{5}t_2\left(\frac{1}{2}(1+\sqrt{5}) \right),$$
and
$$ H(x)+H(x^2)= 42 +15m. $$
Hence $-a_k\equiv 2 $ mod 5. Since $p\equiv 2$ mod $3$ we have $3\nmid \alpha\beta$ and $\alpha^2+5\beta^2\equiv 0$ mod 3 and $\pm b_k\equiv 0$ mod $3$. But $(a_k \text{ mod }5, b_k \text{ mod }3)$ cycles through $(-2,1)$,$(2,0),$ $(-2,-1)$, $(2,0)$ never $(-2,0)$.
So $p^2$ must have a bad representation.
\end{proof}


\begin{thebibliography}{99}
\bibitem{Apostol}
T. M. Apostol, \textit{Resultants of cyclotomic polynomials},
Proc. Amer. Math. Soc. \textbf{24} (1970), 457-462. 



\bibitem{dihedral}
T. Boerkoel \& C. Pinner, \textit{Minimal  group determinants and the  Lind-Lehmer problem for dihedral groups},   Acta Arith. \textbf{186} (2018), no. 4, 377-395.	arXiv:1802.07336 [math.NT].




\bibitem{Stian}
S. Clem and C. Pinner, \textit{The Lind Lehmer constant for 3-groups},  Integers \textbf{18}  (2018), Paper No. A40, 20 pp.

\bibitem{Conrad}
K.\ Conrad, \textit{The origin of representation theory},  Enseign. Math. (2) \textbf{44} (1998), no. 3-4, 361-392.



\bibitem{Lalin}
O. Dasbach and M. Lal\'{i}n,  \textit{Mahler measure  under variations of the base group,} Forum Math. \textbf{21} (2009), 621-637.
 





\bibitem{dilum}
D.\ De Silva and  C. Pinner,  \textit{The Lind-Lehmer constant for $\mathbb Z_p^n$}, Proc. Amer. Math. Soc.  \textbf{142} (2014), no.~6, 1935-1941.

\bibitem{pgroups}
D.\ De Silva, M. Mossinghoff, V. Pigno  and  C. Pinner,  \textit{The Lind-Lehmer constant for certain $p$-groups}, 
 Math. Comp. \textbf{88} (2019), no. 316, 949-972.


\bibitem{Formanek}
E.\ Formanek and D.\ Sibley, \textit{The group determinant determines the group}, Proc. Amer. Math. Soc. \textbf{112} (1991), 649-656.


\bibitem{Norbert}
N.\ Kaiblinger,  \textit{On the Lehmer constant of finite cyclic groups}, Acta Arith. \textbf{142} (2010), no.~1, 79-84.


\bibitem{Norbert2}
N.\ Kaiblinger, \textit{Progress on Olga Taussky-Todd's circulant problem}, Ramanujan J. \textbf{28} (2012), no. 1, 45-60.


\bibitem{Kronecker}
L.\ Kronecker, \textit{Zwei s\"{a}tze \"{u}ber gleichungen mit ganzzahligen coefficienten}, J.Reine Angew. Math. \textbf{53} (1857), 173-175.


\bibitem{Laquer}
H.\ Laquer, \textit{Values of circulants with integer entries},  in  A Collection of Manuscripts Related to the Fibonacci Sequence, pp. 212-217. Fibonacci Assoc., Santa Clara (1980) 


\bibitem{Lehmer}
D.~H.\ Lehmer, \textit{Factorization of certain cyclotomic functions}, Ann. Math. (2) \textbf{34} (1933), no.~3, 461-479.

\bibitem{ELehmer}
E. T. Lehmer, \textit{A numerical function applied to cyclotomy}, Bull. Amer. Math. Soc. \textbf{36} (1930), 291-298.


\bibitem{Lind}
D.\ Lind, \textit{Lehmer's problem for compact abelian groups}, Proc. Amer. Math. Soc. \textbf{133} (2005), no.~5, 1411-1416.

\bibitem{2gp}
M.\ Mossinghoff, V.\ Pigno and C.\ Pinner, \textit{The Lind-Lehmer constant for $\mathbb Z_2^r\times \mathbb Z_4^s$}. Mosc. J. Comb. Number Theory \textbf{8} (2019), no. 2, 151-162.

\bibitem{Newman1}
M.\ Newman, \textit{ On a problem suggested by Olga Taussky-Todd}, Ill. J. Math. \textbf{24} (1980), 156-158.


\bibitem{Newman2}
M.\ Newman, \textit{Determinants of circulants of prime power order}, Linear Multilinear Algebra \textbf{9} (1980), 187-191.  

\bibitem{Pigno1}
V.\ Pigno and  C.\ Pinner, \textit{The Lind-Lehmer constant for cyclic groups of order less than $892,371,480$}, Ramanujan J. \textbf{33} (2014), no.~2, 295--300.

\bibitem{smallgps}
C.\ Pinner and C.\ Smyth, \textit{Integer group determinants for small groups},  Ramanujan J. \textbf{51} (2020), no. 2, 421-453.

\bibitem{Cid2}
C.\ Pinner and  W.\ Vipismakul, \textit{The Lind-Lehmer constant for  $\mathbb Z_{m} \times \mathbb Z^{n}_{p}$}, Integers \textbf{16} (2016), \#A46, 12pp.

\bibitem{S4} 
C.\ Pinner,
\textit{The integer group determinants for the symmetric group of degree four}, Rocky Mountain J. Math. \textbf{49} (2019), no. 4, 1293-1305. 

\bibitem{Simon} D.\ Simon, \textit{Solving norm equations in relative number fields using S-units}, Math. Comp. \textbf{71}
(2002), no. 239, 1287–1305.

\bibitem{Washington} L.\ C.\ Washington, {Introduction to Cyclotomic Fields}, GTM 83, Springer 1982.

\bibitem{Wei}   D.\ Wei, \textit{The unramified Brauer group of norm one tori}, arXiv:1202.4714 [math.NT].

\end{thebibliography}
\end{document}